\documentclass[graybox]{svmult}

\usepackage{color}
\definecolor{darkorchid}{rgb}{0.6,0.196,0.8}

\usepackage{amssymb,amsmath,amsfonts}
\usepackage[all,arc,cmtip]{xy}
\usepackage{tikz}
\usepackage{enumerate}
\usepackage{mathrsfs}
\usepackage{color}
\usepackage{hyperref}

\newcommand{\figref}[1]{Figure~\ref{fig:#1}}
\newcommand{\exref}[1]{Example~\ref{ex:#1}}
\newcommand{\secref}[1]{Section~\ref{sec:#1}}
\newcommand{\thmref}[1]{Theorem~\ref{thm:#1}}
\newcommand{\defref}[1]{Definition~\ref{def:#1}}

\graphicspath{{./}{figs/}}

\usepackage{tikz}
\usepackage{tikz-3dplot}

\usepackage{epsfig,amssymb} 

\newcommand{\di}[1]{\overrightarrow{#1}}

\newcommand{\vecj}{\mathbf{j}}
\newcommand{\vecv}{\mathbf{v}}
\newcommand{\vecw}{\mathbf{w}}

\newcommand{\R}{\mbox{$\mathbb R$}}

\usepackage{subcaption}
\usepackage{type1cm}        
\usepackage{makeidx}         
\usepackage{graphicx}        
\usepackage{multicol}        
\usepackage[bottom]{footmisc}

\usepackage{newtxtext}       %
\usepackage{newtxmath}       

\makeindex             

\begin{document}

\title*{Towards Directed Collapsibility}

\author{Robin Belton,
       Robyn Brooks,
       Stefania Ebli,
       Lisbeth Fajstrup,
       Brittany Terese Fasy,
       Catherine Ray,
       Nicole Sanderson,
       and Elizabeth Vidaurre}
\institute{ Robin Belton \at Montana State University, Bozeman, MT, USA,
        \email{robin.belton@montana.edu},
    \and Robyn Brooks \at Tulane University, New Orleans, LA, USA,
        \email{rbrooks3@tulane.edu},
    \and Stefania Ebli \at EPFL, Lausanne, Switzerland,
        \email{stefania.ebli@epfl.ch},
    \and Lisbeth Fajstrup \at Aalborg University, Aalborg, Denmark
        \email{fajstrup@math.aau.dk},
    \and Brittany Terese Fasy \at Montana State University Bozeman, MT, USA,
        \email{brittany.fasy@montana.edu},
    \and Nicole Sanderson \at Lawrence Berkeley National Lab, Berkeley, CA, USA,
        \email{nsanderson@lbl.gov},
    \and Catherine Ray \at
        Northwestern University, Evanston, IL, USA
        \email{cray@math.northwestern.edu}
    \and Elizabeth Vidaurre \at Molloy College, Rockville Centre, NY, USA,
        \email{evidaurre@molloy.edu}}
%
%
%

\maketitle


\abstract{
In the directed setting, the spaces of directed paths between fixed initial and
    terminal points are the defining feature for distinguishing different
    directed spaces. The simplest case is when the space of directed paths is
    homotopy equivalent to that of a single path; we call this the \emph{trivial
    space of directed paths}. Directed spaces that are topologically trivial may
    have non-trivial spaces of directed paths, which means that information is
    lost when the direction of these topological spaces is ignored. We define a
    notion of  directed collapsibility in the setting of a directed Euclidean
    cubical complex using the {spaces of directed paths} of the underlying
    directed topological space relative to an initial or a final vertex. In
    addition, we give sufficient conditions for a directed Euclidean cubical
    complex to have a contractible or a connected space of directed paths from a
    fixed initial vertex. We also give sufficient conditions for the path space
    between two vertices in a Euclidean cubical complex to be disconnected. Our results 
    have applications to speeding up the verification process of concurrent
    programming and to understanding partial executions in concurrent programs.
}

\section{Introduction}\label{sec:intro}

Spaces that are equipped with a direction have only recently been given more
attention from a topological point of view. The spaces of directed
paths are the defining feature for distinguishing different directed spaces.
 One reason for studying directed spaces is their application to the modeling of concurrent programs where standard algebraic topology does not provide the tools needed~\cite{lisbeth}. Concurrent programming is used when multiple processes need to access shared resources.
Directed spaces are models for concurrent
programs where paths
respecting the time directions represent executions of programs. In such models,
executions are
equivalent if their execution paths are homotopic through a family of directed paths. This observation has already led to new insights and algorithms.  For instance, verification of concurrent programs
is simplified by verifying one execution from each connected component of the space of directed paths; see  \cite{lisbeth} and \cite{FGR}.

While equivalence of executions is clearly stated in concurrent programming, equivalence of
the directed topological spaces themselves is not well understood.  Directed versions of homotopy
groups and homology groups are not agreed upon. Directed homeomorphism is too strong, whereas directed homotopy equivalence is often too
weak, to preserve the properties of the concurrent programs. In classical
(undirected) topology,
the concept of simplifying a space by a sequence of
collapses goes back to J.H.C.\ Whitehead~\cite{whitehead}, and
has been studied in \cite{barmak-minian:2012, forman:2001}, among others.
However, a definition for a directed collapse of a Euclidean
cubical complex that preserves spaces of directed paths is notably missing from
the literature.



In this article, we consider spaces of directed paths in Euclidean cubical
complexes, which corresponds to concurrent programs without loops. Our objects of study are spaces of directed paths relative to a fixed pair of endpoints.  We show how
local information of the past links of vertices in a Euclidean cubical complex
can provide global information on the spaces of directed paths. As an example, our results are applied to study the spaces of directed paths in the well-known
dining philosophers problem.
Furthermore, we define directed collapse so that
a directed collapse of a Euclidean cubical complex preserves the relevant spaces of
directed paths in the original complex. Our theoretical work has applications to
simplifying verification of concurrent programs and better understanding partial
executions in concurrent programs.

In this article we begin in~\secref{motivatingexamples} by illustrating two 
motivating examples of how the execution of concurrent programs can be modeled 
by Euclidean cubical complexes and directed path spaces.
In \secref{pastlinks} we introduce the notions of spaces of directed paths and
Euclidean cubical complexes. Given the directed structure of these Euclidean 
cubical complexes, we do not consider the link of a vertex but the past link of a
vertex. In \secref{sufficient} we give results on the topology of the spaces of 
directed paths from an initial vertex to other vertices in terms of past links.
\thmref{contractibility} gives sufficient conditions on the past links of every
vertex of a complex so that spaces of directed paths are contractible. 
\thmref{connected} gives conditions that are sufficient for the spaces of directed 
paths to be connected. We provide a class of examples that satisfy~\thmref{connected}.  
In \thmref{obstructions} we give sufficient conditions on the past link of a 
vertex so that the space of directed paths from the initial vertex to that vertex 
is disconnected.


In \secref{dircollapse} we describe a method of collapsing one complex into a 
simpler complex while preserving the directed path spaces. This involves taking 
a pair of simplices $(\tau, \sigma)$ from a Euclidean complex $K$ with certain 
conditions on the nearby links and then collapsing $K$ into a simpler complex by 
removing $\tau$, $\sigma$ and all simplices in between.

%
\section{Concurrent Programs and Directed Path Spaces}\label{sec:motivatingexamples}

We illustrate how to organize possible executions of concurrent programs using
Euclidean cubical complexes and directed spaces. An execution is a scheduling of
the events that occur in a program in order to compute a specific task. In~\exref{diningphil} we describe the dining philosophers problem. In~\exref{swissflag}, we illustrate how to model executions of concurrent
programs in the context of the dining philosophers problem when there are two
philosophers.

\begin{example}[Dining Philosophers]\label{ex:diningphil}
The dining philosophers problem originally
formulated by E.\ Dijkstra 
    and
reformulated by T. Hoare~\cite{hoare}
illustrates issues that arise in concurrent programs. Consider $n$ philosophers
sitting at a round table ready to eat a meal. Between each pair of neighboring
philosophers is a chopstick for a total of $n$ chopsticks. Each philosopher must
eat with the two chopsticks lying directly to the left and right of her. Once the
philosopher is finished eating, she must put down both chopsticks. Since there
are only $n$ chopsticks, the philosophers must share the chopsticks in order for
all of them to eat. The dining philosopher problem is to design a concurrent program where all
$n$ philosophers are able to eat once for some amount of time.

In this case, a design of a program is a choice of actions for each
philosopher. One example of a design of a program is where each of the $n$
philosophers does the following:

\begin{itemize}
\item[1.] Wait until the right chopstick is available then pick it up.
\item[2.] Wait until the left chopstick is available then pick it up.
\item[3.] Eat for some amount of time.
\item[4.] Put down the left chopstick.
\item[5.] Put down the right chopstick.
\end{itemize}

This design has states in which every philosopher has
picked up the chopstick to her right and is waiting for the other chopstick.
Such a situation exemplifies a \emph{deadlock} in concurrent programming, an execution that gets "stuck" and never finishes.

This design of a program also has states that cannot occur. For simplicity, consider the dining philosophers problem when $n=2$. The state in which both philosophers are finished eating and one is still holding onto chopstick $a$ while the other is holding chopstick $b$ would imply that a philosopher was able to eat with only one chopstick. This is an example of an \emph{unreachable} state in
concurrent programming.

Lastly, there are executions of this program design that are correct. Any execution in
which the philosophers take turns eating alone is correct since each philosopher will finish.

The dining philosophers problem illustrates the difficulties in
designing concurrent programs. Difficulties arise since each philosopher
must use chopsticks that must be shared with the neighboring philosophers.
Analogously, in concurrent programming, multiple processes must access shared
resources that have a finite capacity.
\end{example}

The next example illustrates how to model executions of the dining philosophers problem with a Euclidean cubical complex. When there are two philosophers this is often referred
to as the \emph{Swiss Flag}.

\begin{example}[Swiss Flag]\label{ex:swissflag}
In the language of concurrent programming the two philosophers represent two processes denoted by $T_1$ and $T_2$. The two chopsticks represent shared resources denoted by
$a$ and $b$. One process is executing the program $P_aP_bV_bV_a$ and the other
process is executing the program $P_bP_aV_aV_b$. Here, $P$ means that a process
has a \emph{lock} on that resource while $V$ means that a process
\emph{releases} a resource. To model this concurrent program with a Euclidean
cubical complex, we construct a $5\times 5$ grid where the $x$-axis is labeled
by $P_aP_bV_bV_a$, each a unit apart, and the $y$-axis is labeled by $P_bP_aV_aV_b$,
each also a unit apart (see~\figref{swiss}). The region $[1,4]\times [2,3]$
represents when both $T_1$ and $T_2$ have a lock on~$a$. In
the dining philosophers problem, a single chopstick can only be held by
one philosopher at a given time. This translates to the shared resources,
$a$ and $b$, each having \emph{capacity} one, where the capacity of a resource
is the number of processes that can have access to the resource simultaneously.
We call the region $[1,4]\times [2,3]$ \emph{forbidden} since $T_1$ and $T_2$
cannot have a lock on $a$ at the same time. The region $[2,3]\times [1,4]$ represents when
both $T_1$ and $T_2$ have a lock on~$b$. This region is also forbidden. The set
complement of the interior of $[1,4]\times [2,3] \cup [2,3]\times[1,4]$ in
$[0,5]\times[0,5]$ is called the Swiss flag and is the Euclidean cubical complex
modeling this program design for the dining philosophers problem.

In general, the Euclidean cubical complex modeling a concurrent program is the
complement of the interior of the forbidden region. An execution is a directed
path from the initial point to the terminal point. Executions are equivalent if
they give the same output given the same input. In geometric terms this means
that corresponding paths are dihomotopic in the path space. In the Swiss flag
there are two distinct directed paths up to homotopy equivalence: one
corresponding to $T_1$ using the shared resources first, and the other
corresponding to $T_2$ using the shared resources first. See~\figref{swiss}.
\end{example}

\begin{figure}
    \centering
    \includegraphics[scale=1.1]{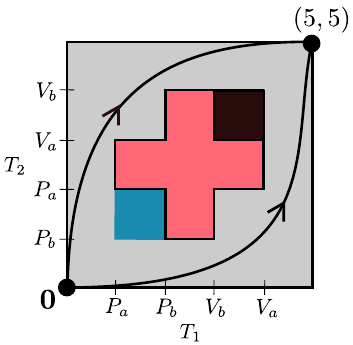}
    \caption{The Swiss Flag. The set of all executions of two
    processes, $T_1$ and~$T_2$ is called the
    \emph{state-space}.
    The pink region $F$ is the
    forbidden region.
    Any bi-monotone path outside of $F$ is a possible execution. There are two regions in the state space that are of
particular interest. The black region is the set of all unreachable states. The blue region is the set of all states that are doomed. A state is doomed if any path starting at that state leads to a deadlock. The black curves in the figure are two possible paths in this directed space.
}\label{fig:swiss}
\end{figure}

\section{Past Links as Obstructions}\label{sec:pastlinks}

In this section, we introduce the notions of spaces of directed paths and
Euclidean cubical complexes. The (relative) past link of a vertex of a Euclidean
cubical complex is defined as a simplicial complex. Studying the contractibility
and connectedness of past links gives us insight on the contractibility and
connectedness of certain spaces of directed paths.
\begin{definition}[d-space] \label{def:dspace}
    A \emph{d-space} is a pair $(X, \di{P}(X))$, where
    $X$ is a topological space and $\di{P}(X) \subseteq P(X):=X^{[0,1]}$
    is a family of paths on $X$ (called \emph{dipaths}) that
    is closed under non-decreasing reparametrizations and concatenations, and
    contains all constant paths.

    For every $x,y$ in $X$, let $\di{P}_x^y(X)$ be the family of \emph{dipaths
    from $x$ to $y$}:

    $$\di{P}_x^y(X) :=\{\alpha \in \di{P}(X): \alpha(0)=x \text{ and } \alpha(1)=y\}.$$

\end{definition}
In particular, consider the following directed space: the \emph{directed real line} $\di{\mathbb{R}}$ is the directed space constructed from the real line whose family of dipaths $\di{P}(\mathbb{R})$  consists of all non-decreasing paths. The \emph{Euclidean space} $\di{\mathbb{R}^n}$ is the $n$-fold product $\di{\mathbb{R}}\times\cdots\times \di{\mathbb{R}}$ with family of dipaths the $n$-fold product $\di{P}(\mathbb{R}^n)=\di{P}(\mathbb{R})\times\cdots\times \di{P}(\mathbb{R})$.

Furthermore, we can solely focus on the family of dipaths in a d-space and endow
it with the compact open topology.

\begin{definition}[Space of Directed Paths]\label{def:dps}
In a d-space $(X, \di{P}(X))$, the \emph{space of directed paths} from~$x$ to $y$ is the family $\di{P}_x^y(X)$ with the compact open topology.
\end{definition}
By topologizing the space of directed paths, we may now use
topological reasoning and comparison. Since $\di{P}_{x}^{y}(X)$ does not have directionality,
contractibility and other topological features are defined as in the
classical case. Moreover, observe that the set~$\di{P}_{x}^{y}(X)$ might have cardinality of the continuum, but is
considered trivial if it is homotopy equivalent to a~point.

The d-spaces that we consider in this article are constructed from Euclidean cubical complexes. Let $\mathbf{p}=(p_1,\dots,p_n),
\mathbf{q}=(q_1,\dots, q_n)\in \mathbb{R}^n$. We write $\mathbf{p}\preceq
\mathbf{q}$ if and only if $p_i \leq q_i$ for all $i=1,\dots, n$.
Furthermore, we denote by $\mathbf{q} -
\mathbf{p}:=(q_1-p_1,\dots,q_n-p_n)$ the component-wise difference between
$\mathbf{q}$ and $\mathbf{p}$, $|\mathbf{p}|:=\sum_{i=1}^n p_i$ is the
element-wise sum, or one-norm, of $\mathbf{p}$. Similarly to the one-dimensional case, the interval
$[\mathbf{p},\mathbf{q}]$ is defined as $\{\mathbf{x}\in\mathbb{R}^n :
\mathbf{p}\preceq \mathbf{x}\preceq \mathbf{q}\}$.

\begin{definition}[Euclidean Cubical Complex]
    Let $\mathbf{p},\mathbf{q}\in \mathbb{R}^n$. If
    $\mathbf{q},\mathbf{p}\in\mathbb{Z}^n$ and $\mathbf{q}-\mathbf{p}\in
    \{0,1\}^n$, then the interval $[\mathbf{p},\mathbf{q}]$ is an \emph{elementary cube in $\mathbb{R}^n$} of dimension $|\mathbf{q}-\mathbf{p}|$. A  \emph{Euclidean cubical complex} $K \subseteq \mathbb{R}^n$ is the union of elementary cubes.
\end{definition}
\begin{remark} A Euclidean cubical complex $K$ is a subset of $\mathbb{R}^n$ and
    it has an associated abstract cubical complex. By a slight abuse of
    notation, we do not distinguish these.
\end{remark}
Every cubical complex $K$ inherits the directed structure from the Euclidean space
$\di{\mathbb{R}^n}$, described after \defref{dspace}. An elementary cube of dimension $d$ is called a $d$-cube.
The  $\emph{m}$-skeleton of $K$, denoted by $K_m$, is the union of all elementary
    cubes contained in~$K
$ that have dimension less than or equal to $m$. The elements of the zero-skeleton are called the vertices
of $K$. A vertex $\mathbf{w}\in K_0$ is said to be \emph{minimal} (resp.,
    \emph{maximal})  if $\mathbf{w} \preceq \mathbf{v}$  (resp.,~$\mathbf{w}
    \succeq \mathbf{v}$) for
every vertex $\mathbf{v}\in K_0$.

Following \cite{ziemianski2016execution}, we define the (relative) past link of a vertex of a Euclidean cubical complex as a simplicial complex.
Let $\Delta^{n-1}$ denote the complete simplicial complex with vertices
$\{1,\dots,n\}$. Simplices of~$\Delta^{n-1}$ is  be identified with elements
$\mathbf{j}\in \{0,1\}^n$. That is, every subset $S \subseteq \{1,\dots,n\}$ is
mapped to the $n$-tuple with entry $1$ in the $k$-th position if~$k$ belongs to
$S$ and $0$ otherwise. The topological space associated to the simplicial
complex~$\Delta^{n-1}$ is the one given by its geometric~realization.

\begin{definition}[Past Link]\label{def:pl} In a Euclidean cubical complex $K$ in $\mathbb{R}^n$, the
    \emph{past link}, $lk^-_{K,\vecw}(\vecv)$, of a vertex $\mathbf{v}$ with respect to
    a vertex $\mathbf{w}$ is the simplicial subcomplex of $\Delta^{n-1}$ defined
    as follows: $\mathbf{j}\in lk^-_{K,\vecw}(\vecv) $ if and only if
    $[\mathbf{v-j},\mathbf{v}]\subseteq K\cap
    [\mathbf{w},\mathbf{v}]$.
\end{definition}

\begin{remark}
    While $K$ is a \emph{cubical} complex, the past link of a vertex in $K$ is always a
    \emph{simplicial}~complex.
\end{remark}

\begin{remark}
{Often the vertex $\mathbf{w}$ and the complex $K$ are understood. In this case we  denote the past link of $\mathbf{v}$ by $lk^-(\vecv)$.
}
\end{remark}

\begin{remark}
{Other definitions of the (past) link are
found in the literature. Unlike \defref{pl}, these are usually subcomplexes of $K$.
However, they are homeomorphic to the (past) link of \defref{pl}.}
\end{remark}

In the following example, we show that vertices of a Euclidean
cubical complex exist that have different past links with respect to two different
vertices. We consider as a Euclidean cubical complex the open top
box
(\figref{PastLinksOpenBox}) and the past links of the vertex $\mathbf{v} = (1,1,1)$ with respect to the vertices $\mathbf{w}=\mathbf{0}$ and $\mathbf{w'}=(0,0,1)$.

\begin{example}[Open Top Box]\label{ex:otb}

    Let $L \subset \mathbb{R}^3$ be the Euclidean cubical complex consisting of
    all of the edges and vertices in the elementary cube $[\mathbf{0},
    \vecv]$ and five of the six two-cubes, omitting the elementary two-cube
    $[(0,0,1), \mathbf{v}]$, i.e., the top of the box.  Because the elementary
    one-cube $
    [\mathbf{v} - (0,0,1), \mathbf{v}]
    \subseteq L \cap [\mathbf{0}, \mathbf{v}] = L$, $lk^-_{L,\mathbf{0}}(\vecv)$ contains the vertex in $\Delta^{2}$
    corresponding to $\mathbf{j} = (0,0,1)$.  Similarly, because the elementary
    two-cube $
    [\mathbf{v}- (0,1,1), \mathbf{v}] \subseteq L$, $lk^-_{L,\mathbf{0}}(\vecv)$ contains the edge
    in $\Delta^{2}$ corresponding to   $\mathbf{j} = (0,1,1)$.  However, because the elementary
    two-cube $[\mathbf{v}- (1,1,0), \mathbf{v}]$ is not contained in $L$,  $lk^-_{L,\mathbf{0}}(\vecv)
    $ does not include the edge corresponding to $\mathbf{j} = (1,1,0)$.  Instead taking the 	   initial vertex to be
    $\mathbf{w} = (0,0,1)$, we get that $lk^-_{L,\mathbf{w}}(\vecv)$ consists of the two
    vertices corresponding to $\mathbf{j} = (0,1,0)$ and $\mathbf{j'} = (1,0,0)$. See
    \figref{PastLinksOpenBox}.

\begin{figure}[h]
    \centering
    {\includegraphics[height=1in]{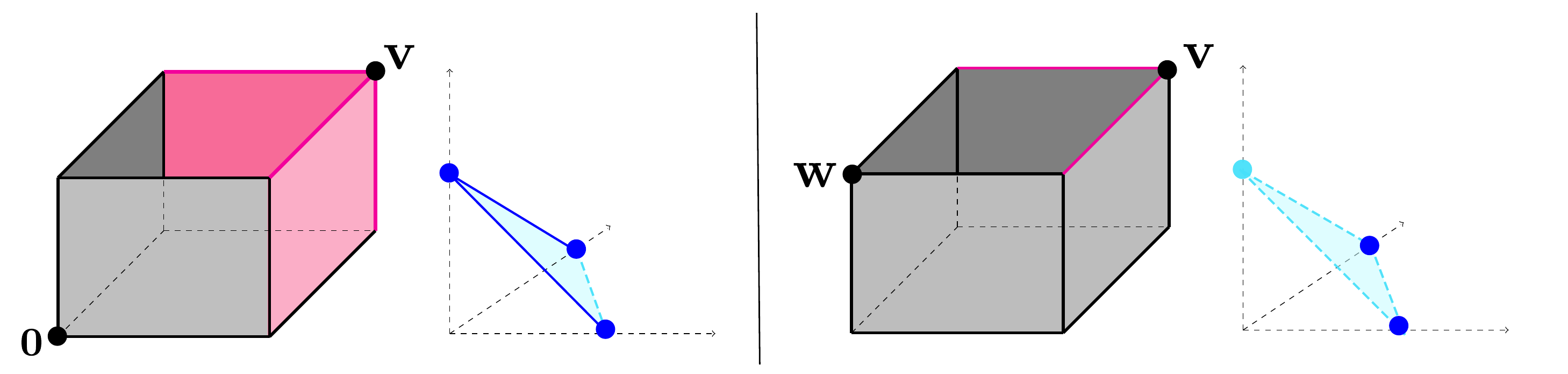}}
    \caption{The Open Top Box. Left: the open top box and the geometric realization of the past link of the
    red vertex $\mathbf{v}=(1,1,1)$ with respect to the black vertex $\mathbf{0}$. The geometric
    realization of the
    simplicial complex $lk^-_{L,\mathbf{0}}(\vecv)$ contains two edges of a triangle,
    since the two red faces are included in~$[\mathbf{0}, \mathbf{v}]$ and three vertices, since
    the three red edges are included in~$[\mathbf{0}, \mathbf{v}]$. Right: the
    open top box and the geometric realization of the past link of the red vertex
    $\mathbf{v}=(1,1,1)$ with respect to the black vertex $\mathbf{w}=(0,0,1)$. The
    geometric realization of the simplicial complex $lk^-_{L,\mathbf{w}}(\vecv)$ consists
    only of two vertices of a triangle, since the two red edges are included in~$[\mathbf{w},
    \mathbf{v}]$.}
    \label{fig:PastLinksOpenBox}
\end{figure}

\end{example}

\section{The Relationship Between Past Links and Path Spaces}\label{sec:sufficient}

In this section, we illustrate how to use past links to study spaces of directed
paths with an initial vertex of $\mathbf{0}$. In particular, the contractibility
and connectedness of all past links guarantees the contractibility and connectedness
of spaces of directed paths. We also provide a partial converse to the
result concerning connectedness.


\begin{theorem}[Contractibility]\label{thm:contractibility} Let $K\subset \R^n$ be a Euclidean
    cubical complex with minimal vertex $\mathbf{0}$. Suppose for
    all $\mathbf{k}\in K_0$, the past link $lk^{-}_{\mathbf{0}}(\mathbf{k})$ is contractible or empty. Then, all
    spaces of directed paths $\di{P}_{\mathbf{0}}^\mathbf{k}(K)$ are
    contractible.
\end{theorem}

\begin{proof} 
 By \cite[Proposition 5.3]{ziemianski2016execution}, if
    $\di{P}_{\mathbf{0}}^\mathbf{k-j}(K)$ is contractible for all $\mathbf{j}\in
    \{0,1\}^n$, $\mathbf{j}\neq \mathbf{0}$, and~$\mathbf{j}\in lk^{-}(\mathbf{k})$, then $\di{P}_{\mathbf{0}}^\mathbf{k}(K)$ is homotopy equivalent to $lk^{-}(\mathbf{k})$. Hence, it suffices to see that all these spaces are contractible.
This follows by structural induction on the partial order on vertices in $K$. The start is at
$\di{P}_\mathbf{0}^{\mathbf{0}+\mathbf{e}_i}(K)$, where $\mathbf{e}_i$ is the
    $i$-th unit vector, and $\mathbf{0}+\mathbf{e}_i\in K_0$.
    If~$[\mathbf{0},{\mathbf{0}+\mathbf{e}_i}]\in K$, then
    $\di{P}_\mathbf{0}^{\mathbf{0}+\mathbf{e}_i}(K)$ is contractible. Otherwise,
    it is empty and the corresponding~$\mathbf{j}$ is not in the past link.
\end{proof}

Now, we give an analogous sufficient condition for when spaces of directed paths
are connected. We provide two different proofs of~\thmref{connected}. The first
proof shows how we can use~\cite[Prop.\ 2.20]{raussen_2000} to get our desired
result. The second proof uses notions from category theory and is based on the
fact that the colimit of connected spaces over a connected category is connected.

\begin{theorem}[Connectedness]\label{thm:connected}
    With $K$ as above, suppose all past links $lk^{-}_{\mathbf{0}}(\mathbf{k})$  of all vertices are connected. Then, for all $\mathbf{k}\in K_0$, all spaces of directed paths $\di{P}_{\mathbf{0}}^\mathbf{k}(K)$ are connected.
\end{theorem}

In this first proof we show that~\cite[Prop.\ 2.20]{raussen_2000} is an equivalent condition to all past links being connected.

\begin{proof}
    In~\cite[Prop.\ 2.20]{raussen_2000},  a local condition is given that
    ensures
    that all spaces of directed paths to a certain final point are connected. Here, we explain how the local
    condition is equivalent to all past links being connected.
    Their condition is in terms of the local future; however, we reinterpret
    this in
    terms of local past instead of local future.
     Since we consider all spaces of directed paths \emph{from} a point (as opposed to \emph{to} a point), this is the right setting we should look at. The
    local condition is the following: for each vertex, $\vecv$ and all
    pairs of edges $[\vecv-\mathbf{e}_r,\vecv]$, $[\vecv-\mathbf{e}_s,\vecv]$ in $K$, there is a sequence of
    two-cells $[\vecv-\mathbf{e}_{k_i}-\mathbf{e}_{l_i},\vecv]$, $i=1,\ldots , m$ all in $K$ s.t.
    $l_i=k_{i+1}$ $i=1,\ldots, m-1$, $k_1=r$ and $l_m=s$.
    Now, we show that this local condition is equivalent to ours.
    In the past link
    considered as a simplicial complex, such a sequence of two-cells corresponds to a sequence of edges
    from the vertex $r$ to the vertex $s$. For $x,y\in lk^-(\vecv)$,
  they are both connected to a vertex via a line. And those vertices
  are connected. Hence, the past link is connected.

    Vice versa: Suppose $lk^-(\vecv)$ is connected. Let $p,q$ be vertices in
    $lk^-(\vecv)$ and let $\gamma: I\to lk^-(\vecv)\in \Delta^{n-1}$ be a path
    from $p$ to $q$. The sequence of simplices traversed by $\gamma$,
    $S_1,S_2,\ldots,S_k$, satisfies $S_i\cap S_{i+1}\neq\emptyset$. Moreover,
    the intersection is a simplex. Let $p_i\in S_i\cap S_{i+1}$. A sequence of
    pairwise connected edges connecting $p$ to $q$ is constructed by such
    sequences from $p_i$ to $p_{i+1}$ in $S_{i+1}$ thus providing a sequence of
    two-cells similar to the requirement in \cite{raussen_2000}.
    Hence, by \cite{raussen_2000}, if all past links of all vertices are connected, then all $\di{P}_{\mathbf{0}}^\mathbf{k}$ are connected
\end{proof}

This second proof of \thmref{connected} has a more categorical flavor.

\begin{proof}
   We give a more categorical argument which is closer to the proof of \thmref{contractibility}.
    In \cite[Prop 2.3 and Equation 2.2]{RZ}, the space of directed paths $\di{P}_{\mathbf{0}}^\mathbf{k}$ is given as a colimit over $\di{P}_{\mathbf{0}}^\mathbf{k-j}$. The indexing category is $\cal{J}_K$ with objects
$\{\mathbf{j}\in\{0,1\}^n: [\mathbf{k}-\mathbf{j}]\subseteq K\}$ and morphisms $\mathbf{j}\to\mathbf{j'}$ for $\mathbf{j}\geq\mathbf{j'}$ given by inclusion of the simplex $\Delta^{\mathbf{j}}\subset\Delta^{\mathbf{j'}}$. The geometric realization of the index category is the past link which with our requirements is connected. The colimit of connected spaces over a connected category is connected. Hence, by induction as above, beginning with edges from $\mathbf{0}$,  $\di{P}_{\mathbf{0}}^\mathbf{k-j}$ are all connected and the conclusion follows.
\end{proof}

\begin{remark} Our conjecture is that similar results for $k$-connected past links should follow from the $k$-connected Nerve Lemma.
\end{remark}

\begin{remark}
   Note that the statements of both \thmref{contractibility} and \thmref{connected} concern past links and path spaces defined with respect to a fixed initial vertex. Past links depend on their initial vertex.  As an example, consider the open top box, \exref{otb} . Then all past links in $L$ with
    respect to the initial vertex $\mathbf{0}$ are contractible,
    but $\di{P}_{\mathbf{w'}}^{\mathbf{v}}(L)$, where $\vecw'=(0,0,1)$ and
    $\vecv=(1,1,1)$, is not contractible. It is in fact two points. Note, this
    does not contradict \thmref{contractibility}, which
    only asserts that $\di{P}_{\mathbf{0}}^{\mathbf{v}}(L)$ is~contractible. See \figref{PastLinksOpenBox}.
\end{remark}

We now show how~\thmref{contractibility} and~\thmref{connected} can be used to
study the spaces of the directed paths in slight modifications of the dining
philosophers problem.

\begin{example}[Three Concurrent Processes Executing the Same Program]\label{ex:diningphil_same}
We consider a modification of \exref{diningphil} where we have three processes and two
resources each with capacity two. All processes are executing the program $P_aP_bV_bV_a$. The Euclidean
cubical complex modeling this situation has three dimensions, each representing the program of a process. Since each resource has
capacity two, it is not possible to have a three way lock on any of them. The three processes have a lock on $a$ in the region $[P_a,V_a]^{\times 3}$,
which is the cube $[(1,1,1),(4,4,4)]$. Similarly, the three processes have a
lock on $b$ in the region $[P_b,V_b]^{\times 3}$ which is the cube
$[(2,2,2),(3,3,3)]$. The forbidden region is the union of these two
sets which is $[(1,1,1),(4,4,4)]$. We can model this concurrent program as
a three-dimensional Euclidean cubical complex and the forbidden region is the inner $3\times 3\times 3$ cube.

In order to analyze the connectedness and contractibility of the spaces of
directed paths with initial vertex $\mathbf{0}$, we study the past
links of the vertices of $K$. First, we
show that not all past links are contractible. Let $\vecv=(4,4,4)$. Then,
$lk^-_{K,\textbf{0}}(\mathbf{v})$ consists of all $\mathbf{j}\in \{0,1\}^3$
except $(1,1,1)$. This is because the cube $[(3,3,3),(4,4,4)]$ is not contained
in $K$, but $[\vecv-\mathbf{j},\vecv]\subset K$ for all other $\mathbf{j}$.
Therefore, $lk^-_{K,\textbf{0}}(\vecv)$ is the boundary of the two simplex (see~\figref{diningphil_same}). Because the boundary of the two simplex is not contractible, the hypothesis
of~\thmref{contractibility} is not satisfied. Hence, we cannot
use~\thmref{contractibility} to study the contractibility of the spaces of
directed paths. 

\begin{figure}[h]
    \centering
    {\includegraphics[height=1.5in]{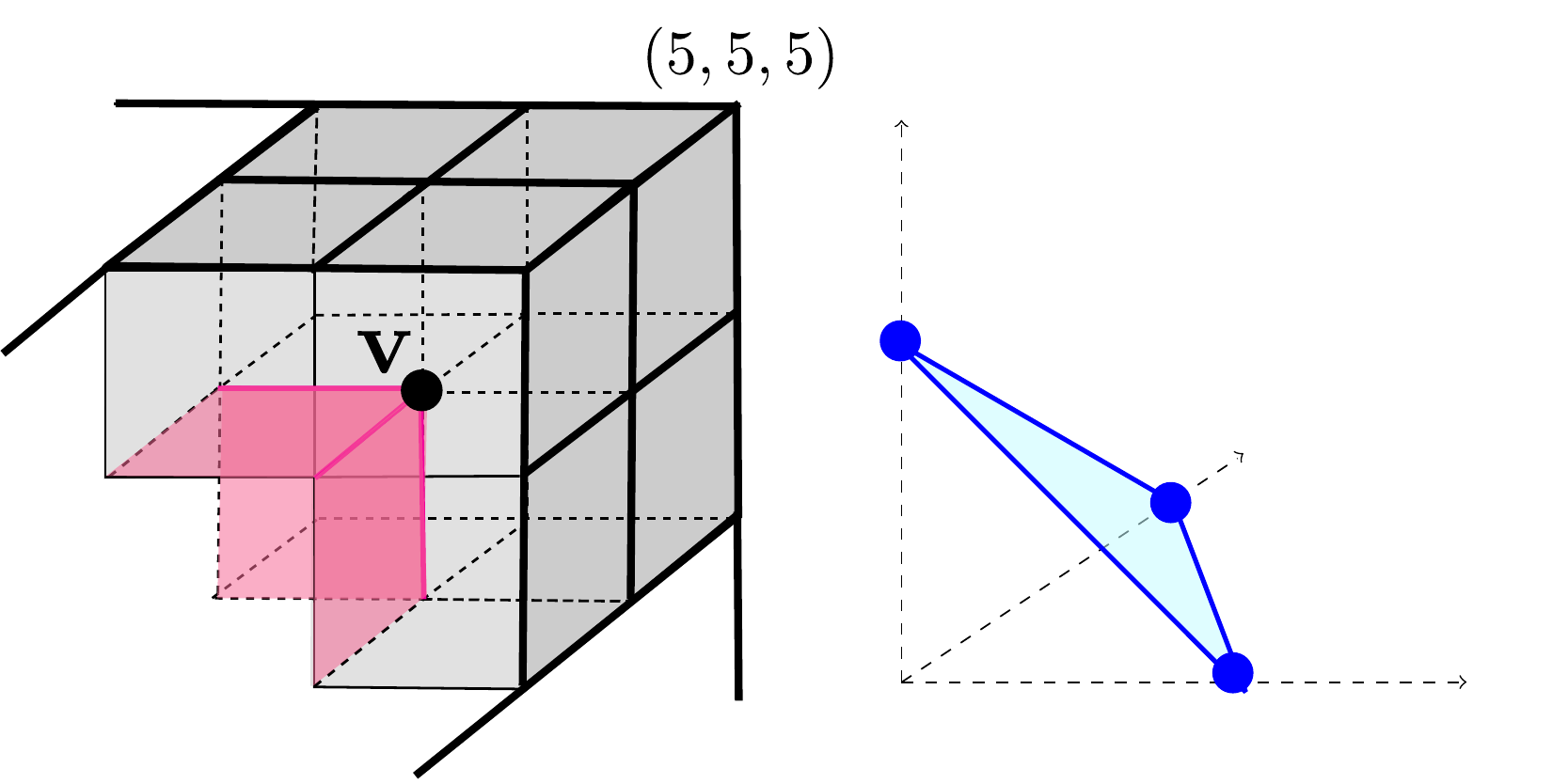}}
    \caption{Three processes, same program. Illustrating $lk^-_{K,\mathbf{0}}(\mathbf{v})$ where $K$ is the cube $[\mathbf{0},(5,5,5)]$ minus the inner cube, $[(1,1,1),(4,4,4)]$, and $\vecv=(4,4,4)$. The geometric realization of the simplicial complex $lk^-_{K,\textbf{0}}(\mathbf{v})$ is the boundary of the two simplex since the three pink faces and edges are included in $[\mathbf{0},\vecv]$.}
    \label{fig:diningphil_same}
\end{figure}

Next, we show that all past links are connected. If we directly compute the past link~$lk^-_{K,\mathbf{0}}(\mathbf{k})$ for all $\mathbf{k}\in K_0$,
we find that the past link consists of either a zero simplex, one simplex, the
boundary of the two simplex, or a two simplex. All these past links are
connected. \thmref{connected} implies that for all $\mathbf{k}\in K_0$, the space
of directed paths,~$\di{P}_{\mathbf{0}}^{\mathbf{k}}(K)$ is connected.

We can generalize this example to $n$ processes and two resources
with capacity $n-1$ where all processes are executing the program $P_aP_bV_bV_a.$
For all $n$~\thmref{connected} shows
that all spaces of directed paths are connected.
\end{example}

The converse of~\thmref{connected} is not true.  To see this, and give the conditions under which the converse does hold, we need to introduce the following definition:

\begin{definition}[Reachable]\label{def:reachable}
The point $x\in K$ is \emph{reachable} from $\mathbf{w}\in K_0$ if there is a path from~$\mathbf{w}$ to~$x$. A subcomplex of $K$ is induced by the set of points that are reachable from a vertex $\mathbf{w}$.
\end{definition}

\begin{example}[Boundary of the $3\times 3\times 3$ Cube with Top Right Cube]\label{ex:reachability}
 Let $K$ be the Euclidean cubical complex that is the boundary of the $3\times 3\times 3$ cube along with the cube $[(2,2,2),(3,3,3)]$. Observe that all spaces of directed paths with initial vertex $\mathbf{0}$ are connected. However, $K$ has a disconnected past link at $\mathbf{v}=(3,2,2)$.  If we consider the subcomplex $\hat{K}$ that is reachable from $\mathbf{0}$, then $\hat{K}$ is the boundary of the $3\times 3\times 3$ cube. The past links of all vertices in $\hat{K}$ are connected. 
This motivates the conditions given in~\thmref{obstructions} of removing the unreachable points of a Euclidean cubical complex.  The connected components of a disconnected past link in the remaining complex can then be represented by directed paths from the initial point and not only locally.

\begin{figure}[h]
    \centering
    {\includegraphics[height=1.3in]{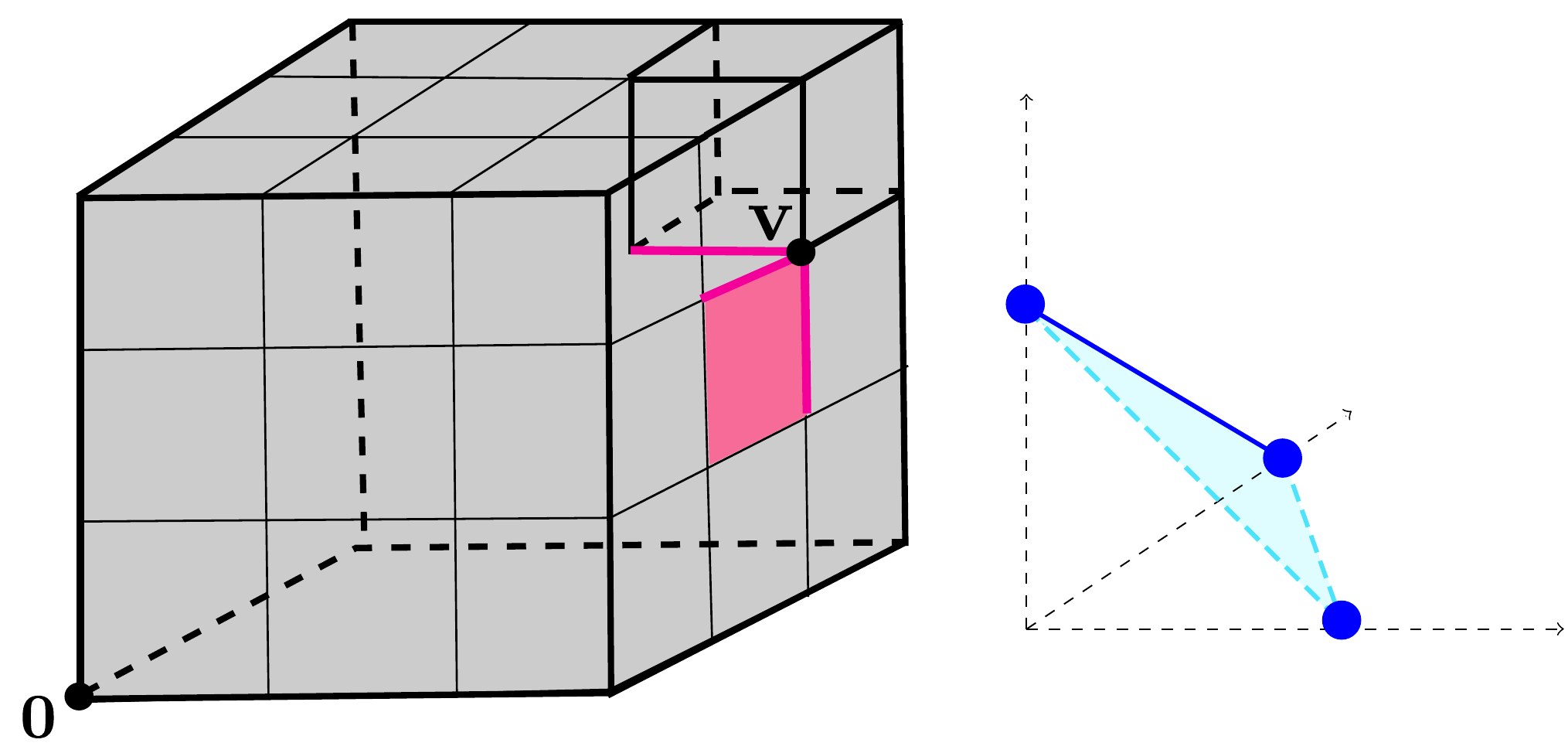}}
    \caption{Motivating reachability condition. For $K$ equal to the boundary of the $3\times 3\times 3$ cube union with $[(2,2,2),(3,3,3)]$, the geometric realization of the simplicial complex $lk^-_{K,\textbf{0}}(\mathbf{v})$ is an edge and a point since the three pink edges and one face are included in $[(0,0,0),\vecv]$. }
    \label{fig:reachex}
\end{figure}

\end{example}

\begin{theorem}[Realizing Obstructions]\label{thm:obstructions}
Let $K$ be a Euclidean cubical complex with initial vertex
$\mathbf{0}$. Let $\hat{K}\subset K$ be the subcomplex reachable from $\mathbf{0}$. If for $\mathbf{v}\in \hat{K}_0$, the past link in $\hat{K}$ is disconnected, then the path space $\di{P}_{\textbf{0}}^{\mathbf{v}}(K)$ is disconnected.
\end{theorem}

\begin{proof}

    Let $\mathbf{v}$ be a vertex such that $lk^-_{K,\textbf{0}}(\mathbf{v})$ is disconnected and let $\mathbf{j}_1, \mathbf{j}_2$ be vertices in $lk_{\hat{K}}^-(v)$ in different components. The edges $[\mathbf{v}-\mathbf{j}_i, \mathbf{v}]$ are then in $\hat{K}$ and, in particular, $\mathbf{v}-\mathbf{j}_i \in \hat{K}_0$. Hence, there are paths $\mu_i: \overrightarrow{I}\to \hat{K}$ such that $\mu_i(0)=\textbf{0}$ and $\mu_i(1)=\mathbf{v}-\mathbf{j}_i$.

By \cite{Fajdianddi}, there are $\hat{\mu}_i$ which are dihomotopic to $\mu_i$ and such that $\hat{\mu}_i$ is combinatorial, i.e., a sequence of edges in $\hat{K}$. Let $\gamma_i$ be the concatenation of $\hat{\mu}_i$ with the edge $[\mathbf{v}-\mathbf{j}_i,\mathbf{v}]$.

Suppose for contradiction that $\gamma_1$ and $\gamma_2$ are connected by a path in
    $\di{P}_{\textbf{0}}^{\mathbf{v}}(K)$. Let $H:\overrightarrow{I}\times I \to
    K$ be such a path with $H(t,0)=\gamma_1(t)$ and $H(t,1)=\gamma_2(t)$. Since
    $H(t,s)$ is reachable from $\textbf{0}$, $H$ maps to~$\hat{K}$.

    By \cite{Fajdianddi}, there is a combinatorial approximation $\hat{H}: \overrightarrow{I}\times I  \to \hat{K}_2$ to the
$2$-skeleton of $\hat{K} \subset K$. Let $B$ be the open ball centered around $\mathbf{v}$ with radius $1/2$. Since $\hat{H}$
is continuous, the inverse image of $B$ under $\hat{H}$ is a neighborhood of $\{1\} \times I \subset \di{I}\times I$.  For $0<\epsilon<1/2$, this neighborhood contains a strip $(1-\epsilon,1] \times I$ (by compactness of $I$). Then $\hat{H}(1-\epsilon/2 \times I)$ gives a
path connecting the two edges $[\mathbf{v}-\mathbf{j}_i,\mathbf{v}]$. This path traverses a sequence of $2$-cubes (the
carriers). These correspond to a sequence of edges in the past link that connect $\mathbf{j}_1$ and $\mathbf{j}_2$, which
contradicts the assumption that they are in different components. Therefore, $\gamma_1$ and $\gamma_2$ correspond to two points
in $\di{P}_{\textbf{0}}^{\mathbf{v}}(K)$ that are not connected by a path.
\end{proof}

In general, the reachability condition in~\thmref{obstructions} eliminates the spurious disconnected past links that could appear in the unreachable parts of a Euclidean cubical complex.

\begin{example}\label{reachableSwissFlag}
 To see how \thmref{obstructions} can be applied, consider \exref{swissflag}, the Swiss flag. There are two vertices with disconnected past links with respect to $\mathbf{0}$ namely $(4,3)$ and $(3,4)$. These disconnected past links imply that  \thmref{connected} is inconclusive.  If the unreachable section of the Swiss flag is removed, we obtain a new Euclidean cubical complex in which the vertex $\vecv=(4,4)$ has a disconnected past link, consisting of two points.  By \thmref{obstructions}, the path space $\di{P}_{\textbf{0}}^{\mathbf{v}}(K)$ is also disconnected. In fact, $\di{P}_{\textbf{0}}^{\mathbf{v}}(K)$ has two points, representing the dihomotopy classes of paths which pass above the forbidden region, and those paths which pass below.  \end{example}

This is an example of the following: given two vertices $\mathbf{w}$ and $\mathbf{v}$ in a Euclidean cubical complex $K$, if the path space $\di{P}_{\textbf{w}}^{\mathbf{v}}(K)$ is disconnected, then it must be that some vertex in $[\mathbf{w},\mathbf{v}]$ has a disconnected past link with respect to $\mathbf{w}$ (the vertices $(4,3)$ and $(3,4)$ in the Swiss flag).  If $\mathbf{w}=\mathbf{0}$, this is the contrapositive of \thmref{connected}. If, moreover, $K$ is reachable from $\mathbf{0}$,  \thmref{obstructions} allows us to draw conclusions about the space of directed paths.

\section{Directed Collapsibility}\label{sec:dircollapse}

To simplify the underlying topological space of a d-space while preserving topological properties of the associated space of directed paths, we introduce the process of directed collapse.  The criteria we require to perform directed collapse on Euclidean cubical complexes involves the topology of the past links of the vertices of the complex.  We defined the past links as simplicial complexes that are not themselves directed, so our topological criteria are in the usual sense.

\begin{definition}[Directed Collapse]\label{def:directedCollapse}
    Let $K$ be a Euclidean cubical complex with initial vertex $\textbf{0}$. Consider $\sigma,\tau\in K$ such that
    $\tau\subsetneq\sigma$, $\sigma$ is maximal, and no other maximal cube
    contains $\tau$.  Let $K'=K\setminus \{\gamma\in
    K|\tau\subseteq\gamma\subseteq \sigma\}$. $K'$  is a \emph{directed
    (cubical) collapse} of $K$  if, for all $\vecv\in K'_0$,
            $lk^-_K(\vecv)$ is homotopy equivalent to $lk^-_{K'}(\vecv)$.
            The pair $\tau,\sigma$ is then called a \emph{collapsing pair}.

            $K'$ is a \emph{directed 0-collapse} of $K$ if for all
            $\vecv\in K_0'$,  $lk^-_K(\vecv)$ is connected if and
            only if  $lk^-_{K'}(\vecv)$ is connected.
\end{definition}
\begin{remark} As in the simplicial case, when we remove $\sigma$ from the
    abstract cubical complex, the effect on the geometric realization is to
    remove the interior of the cube corresponding to~$\sigma$.
\end{remark}
\begin{remark}
Note for finding collapsing pairs, $(\tau, \sigma)$,
    using~\defref{directedCollapse}, with the geometric realization of $\sigma$
    given by the elementary cube, $[\vecw-\vecj,\vecw]$, it is sufficient to
    only check $\vecv\in K'_0$ such that $\vecv=\vecw-\vecj'$ where
    $\vecj-\vecj'>0$. Otherwise the past links, $lk^-_K(\vecv)$ and
    $lk^-_{K'}(\vecv)$, are~equal.
\end{remark}



\begin{definition}[Past Link Obstruction] \label{def:linkobstruction}
    Let $\vecw\in K_0$. A \emph{past link
    obstruction (type-$\infty$)} in~$K$ with respect to $\vecw$ is a vertex $\vecv\in K_0$ such that
    $lk^{-}_{K,\vecw}(\vecv)$ is not contractible. A \emph{past link
    obstruction (type-$0$)}  in $K$ with respect to $\vecw$ is a vertex $\vecv\in K_0$ such that
    $lk^{-}_{K,\vecw}(\vecv)$ is not~connected.

\end{definition}

Directed collapses preserve some topological properties of the space of directed
paths. In particular:

\begin{corollary}
If there are no type-$\infty$ past link obstructions, then all spaces of directed paths from the initial point are contractible.  If there are no type-$0$ past link obstructions, all spaces of directed paths from the initial point are connected.
\end{corollary}

\begin{proof}
Contractibility is a direct consequence of \thmref{contractibility}.  Likewise, connectedness follows from \thmref{connected}.
\end{proof}

\begin{corollary}[Invariants of Directed Collapse]\label{thm:invariant}
    If we have a sequence of directed collapses from $K$ to $K'$, then there are
    no obstructions in $K$ iff there are no obstructions in $K'$.
\end{corollary}


\begin{remark}[Past Link Obstructions are Inherently Local]
The past link of a vertex is constructed using local (rather than global) information from the cubical complex.  Therefore, a past link obstruction is also a local property, which is not dependent on the global construction of the cubical complex. \
\end{remark}


Below, we provide a few motivating examples for our definition of directed collapse.
In general, we want our directed collapses to preserve all spaces of directed paths between the
initial vertex and any other vertex in our cubical complex.  Notice, $\tau$ from \defref{directedCollapse} is a \emph{free face} of $K$.
Performing a directed collapse with an arbitrary free face of a directed space $K$ with
minimal element $\mathbf{0} \in K_0$ and maximal element $\mathbf{1} \in K_0$
can modify the individual spaces of directed paths $\di{P}_{\mathbf{0}}^{\vecv}(K)$
and $\di{P}_{\vecv}^{\mathbf{1}}(K)$ for $\vecv \in K_0$.

When
$\di{P}_{\vecv}^{\mathbf{1}}(K) = \emptyset$, we call $\vecv$ a
\emph{deadlock}.  When $\di{P}_{\mathbf{0}}^{\vecv}(K) = \emptyset$, we call
$\vecv$ \emph{unreachable}. Deadlocks and unreachable vertices are in a sense
each others opposites.  Notice if we take the same directed space~$K$ yet reverse the direction of all dipaths, then deadlocks become unreachable vertices and vice versa.
However, as \exref{grid} and \exref{collveredg} illustrate, the creation of an unreachable vertex in the process of a directed collapse might result in a past link obstruction at a neighboring vertex while the creation of a deadlock does not.

\begin{example}[3 x 3 Grid, Deadlocks \& Unreachability]\label{ex:grid}
Let $K$ be the Euclidean cubical complex in $\mathbb{R}^2$ that is the $3\times 3$
grid. Consider the Euclidean cubical complexes~$K'$ and~$K''$ obtained by
    removing~$(\tau, \sigma)$ with $\tau = [(1,3), (2,3)], \sigma
    =[(1,2),(2,3)]$ and~$(\tau' , \sigma')$ with~$\tau' =
    [(1,0), (2,0)], \sigma' = [(1,0), (2,1)]$,
respectively.  While $K'$  is a directed collapse of~$K$,~$K''$
is not a directed collapse of $K$; this is because $K''$ introduces a
past link obstruction at~$(2,1)$.  So,~$(\tau, \sigma)$ is a collapsing
pair while $(\tau', \sigma')$ is not. Collapsing~$K$ to~$K'$ creates a
deadlock at~$(1,3)$ but this does not change the space of directed paths from the
designated start vertex $\mathbf{0}$ to any of the vertices between $\mathbf{0}$
and the designated end vertex~$(3,3)$ (see $K'$ in \figref{DeadlocksUnreachable}).  However, collapsing~$K$ to~$K''$
creates an unreachable vertex~$(2,0)$ from the start vertex $\mathbf{0}$  (see $K''$ in \figref{DeadlocksUnreachable}) which does change the space of directed paths from $\mathbf{0}$ to $(2,0)$ to be empty. Hence not all spaces of directed paths starting at $\mathbf{0}$ are preserved. This motivates our definition of directed~collapse.
\end{example}

\begin{figure}[h!]
    \centering
    {\includegraphics[width=.8\textwidth]{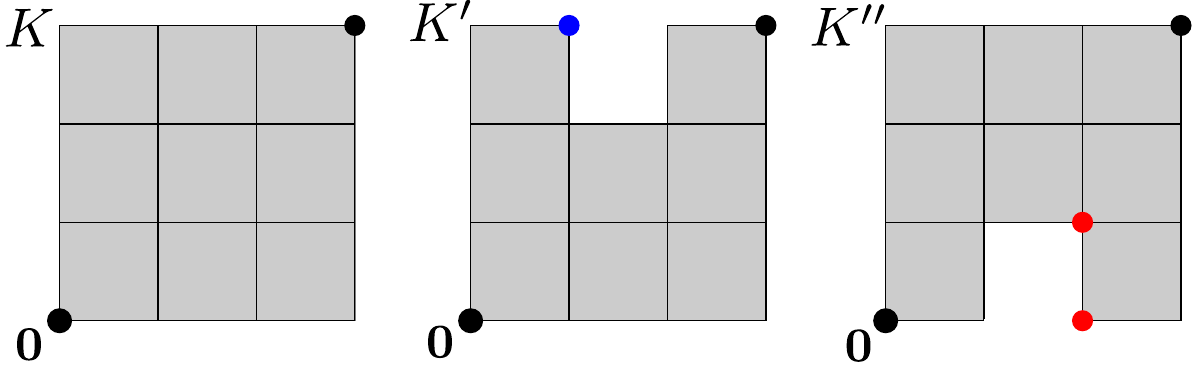}}\qquad
    \caption{Illustrating \exref{grid}. On the left: the cubical complex $K$ with initial vertex $\mathbf{0}$ and final vertex $(3,3)$. In the center: The cubical complex $K'$ which is a directed collapse of $K$. The deadlock in blue does not change the space of directed paths from $\mathbf{0}$ to any of the vertices between $\mathbf{0}$
and~$(3,3)$. On the right: the cubical complex $K''$ which is not a
    directed collapse of $K$. The space of directed paths into the unreachable red vertex,
    $(2,0)$, becomes empty. This is reflected in the topology of the past link of the red vertex $(2,1)$ (see \exref{collveredg}).  }
    \label{fig:DeadlocksUnreachable}
\end{figure}

Our next example shows how directed collapses can be performed with collapsing
pairs~$(\tau, \sigma)$ when $\tau$ is of codimension one and greater.

\begin{example}[3 x 3 grid, Edge \& Vertex Collapses]\label{ex:collveredg}
    Consider again the Euclidean cubical complex $K$ from \exref{grid}.  If we allow
    a collapsing pair~$(\tau, \sigma)$ with $\tau$ of dimension greater than
    $0$, we may introduce deadlocks or unreachable vertices. In particular, collapsing the free edge $\tau = [(1,3), (2,3)]$ of the top blue square $\sigma = [(1,2),(2,3)]$ in \figref{EdgeVertexCollapse} changes the space of directed paths $\di{P}_{(1,3)}^{(3,3)}(K)$ from being trivial to empty in $K \backslash \{ \gamma \vert \tau \subseteq \gamma\subseteq \sigma \}$. Yet we
    care about preserving the space of directed paths from our designated start vertex $\mathbf{0}$
    to any of the vertices $(i,j)$ with $0 \leq i,j \leq 3$ since we ultimately are interested in preserving the path space $\di{P}_{\mathbf{0}}^{(3,3)}(K)$.  Because
    of this, such collapses should be allowed in our directed setting.  Note that, in
    these cases, the past link of all vertices remains contractible.
   However, collapsing the free edge $\tau' = [(1,0), (2,0)]$ of the bottom red square $ \sigma' = [(1,0),(2,1)]$ in
    \figref{EdgeVertexCollapse} changes the path space $\di{P}_{\mathbf{0}}^{(2,0)}(K)$ from being trivial to empty.  This is reflected in the non-contractible past link of $(2,1)$ in $K\backslash \{ \gamma \vert \tau' \subseteq \gamma \subseteq \sigma' \}$  that consists of the two vertices $\mathbf{j} = (1,0)$ and $\mathbf{j'} = (0,1)$ but not the edge $\mathbf{j''} = (1,1)$ connecting them.
    Restricting our collapsing pairs to only include $\tau$ of
    dimension 0 allows for only two potential collapses, the corner vertices~$(0,3)$ and $(3,0)$ into the yellow squares $[(0,2), (1,3)]$ and $[(2,0), (3,1)]$, respectively.  Neither of these collapses
    create deadlocks or unreachable vertices and the contractibility of the past
    link at all vertices is preserved.  Performing these corner vertex collapses
    exposes new free vertices that can be a part of subsequent collapses.

\centering
\begin{figure}[h!]
    \hfill
    \begin{subfigure}[b]{0.4\textwidth}
        \includegraphics[width=.6\textwidth]{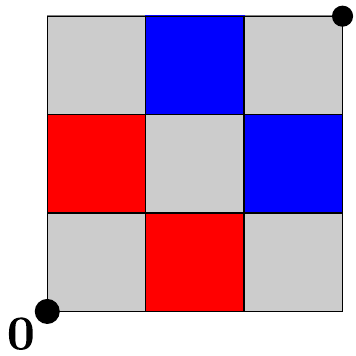}
        \caption{Edge Collapse}
        \label{fig:collapse-edge}
    \end{subfigure}
    \quad 
    \begin{subfigure}[b]{0.4\textwidth}
        \includegraphics[width=.6\textwidth]{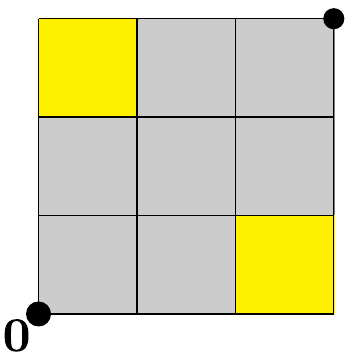}
        \caption{Vertex Collapse}
        \label{fig:collapse-vertex}
    \end{subfigure}
    \caption{Illustrating \exref{collveredg}. On the left: the collapsing of the
    free edge in the blue squares is an admitted directed collapse. The
    collapsing of the free edge in the red squares is not an admitted directed
    collapse. On the right: the collapsing of the free vertex in the yellow
    squares is an admitted directed collapse.}
    \label{fig:EdgeVertexCollapse}
\end{figure}
\end{example}

Lastly, we explain how the Swiss flag can be collapsed using a sequence of $0$-collapses. The Swiss flag contains uncountably many paths between the initial and final vertex. After performing the sequence of 0-collapses as described in~\exref{CollapseSwissflag}, there are only two paths up to reparametrization between the initial and final vertex. These two paths represent the two dihomotopy classes of paths that exists for the Swiss Flag. Referring back to concurrent programming, this means there are only two ways to design a concurrent program; either the first process holds a lock on the two resources then releases them so the other process can place a lock on the resources or vice versa.

\begin{example}[0-collapsing the Swiss Flag] The Swiss flag considered
  as a Euclidean cubical complex in the $5\times 5$ grid has
  vertices with connected past links, except at $(4,3)$ and $(3,4)$. The vertex $(2,2)$ and the
  cube $[1,2]\times [1,2]$ are a $0$-collapsing pair. The vertex
  $(3,3)$ and the cube $[3,4]\times [3,4]$ are not, since that
  collapse would produce a disconnected past link at $(4,4)$. A
  sequence of $0$-collapses preserving the initial and final point will give a
  1-dimensional Euclidean cubical complex and one 2-cube. Specifically, we get the edges
  $[0,1]\times\{0\},$ $\{1\}\times [0,1]$, $\{1\}\times [1,3]$, $[1,3]\times\{1\}$, $[1,2]\times\{3\}$, $\{3\}\times[1,2]$, $\{2\}\times[3,4]$, $[3,4]\times \{2\}$, $[2,3]\times\{4\}$, $\{4\}\times [2,3]$, the square $[3,4]\times [3,4]$, and lastly the edges $\{4\}\times [4,5]$ and $[4,5]\times \{5\}$.
  \label{ex:CollapseSwissflag}
  \end{example}

\begin{figure}[h!]
    \centering
    {\includegraphics[width=.7\textwidth]{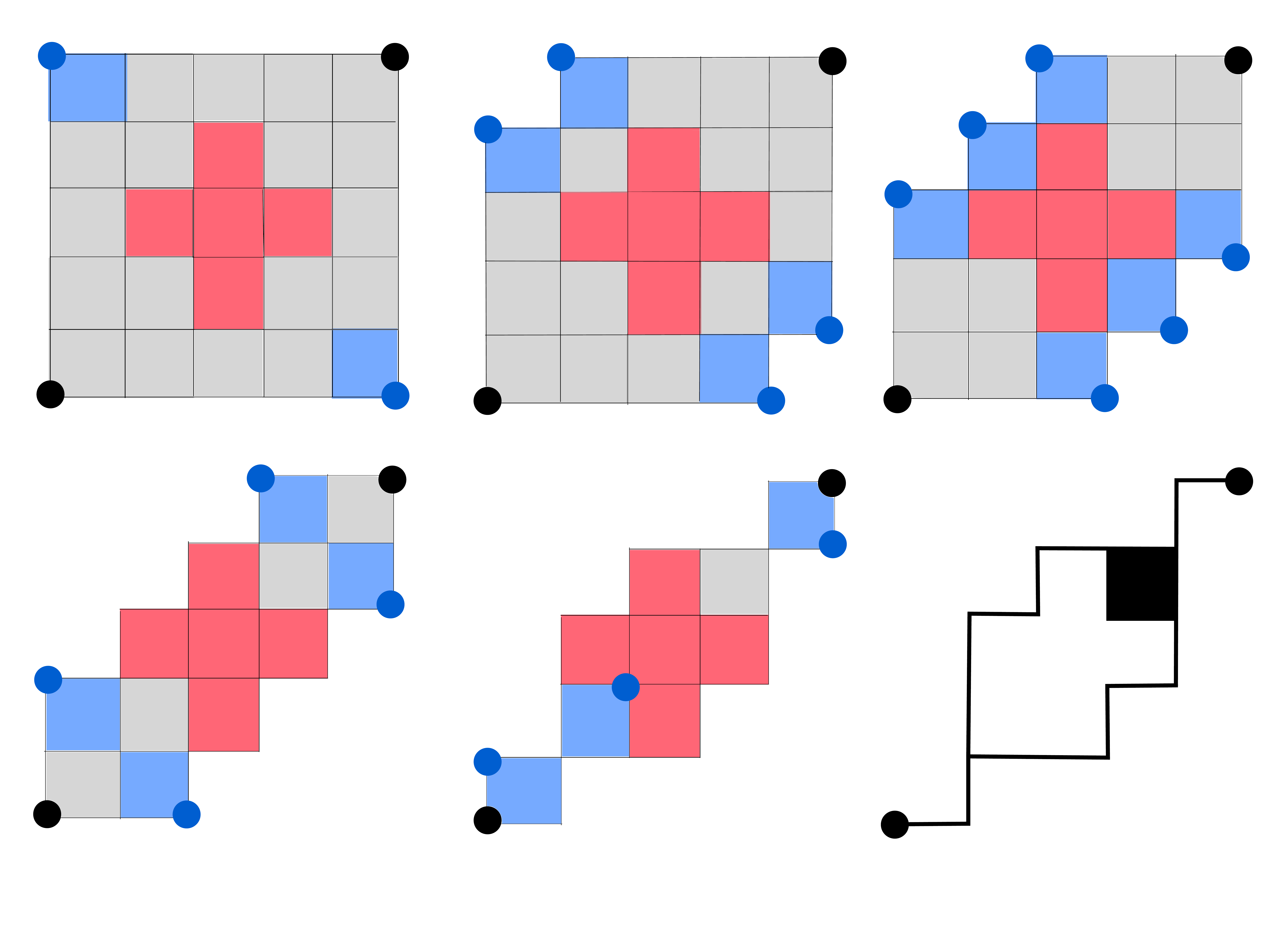}}  \qquad
    \caption{Zero-collapsing the Swiss Flag. A sequence of zero-collapses is
    presented from the top left to bottom right. At each stage, the faces and
    vertices shaded in blue represent the 0-collapsing pairs at that stage. The
    result of the sequence is shown in the bottom right which is a
    one-dimensional Euclidean cubical complex and one two-cube.}
    \label{fig:CollapseSwissflag}
\end{figure}

\section{Discussion}\label{sec:discussion}
Directed topological spaces have a rich underlying structure and many interesting applications.  The analysis of this structure requires tools that are not fully developed, and a further investigation into these methods will lead to a better understanding of directed spaces.  In particular, the development of these notions, such as directed collapse, may lead to a better understanding of equivalence of directed spaces and their spaces of directed paths.

Interestingly, when comparing directed collapse
with the notion of cubical collapse in the undirected case, two main contrasts
arise.  First, the notion of directed collapse is stronger than that of cubical
collapse; any directed collapse is a cubical collapse, but not all cubical
collapses satisfy the past link requirement of directed collapse.  However,
directed collapse is not related to existing notions of dihomotopy equivalence which involve continuous maps between topological spaces that preserve directed paths.
This contrasts the undirected case; any two spaces related by cubical collapses are homotopic.
This suggests the need for dihomotopy equivalence with respect to an initial 
point.


Directed collapse may not preserve dihomotopy equivalence, so we can collapse more than e.g. Kahl. By~\thmref{invariant}, if $K'$ is a directed collapse of~$K$ with respect to~$\vecv$ and~$K'$ has trivial spaces of directed paths from~$\vecv$, then so does $K$. Similarly, if all spaces of
directed paths are connected in $K'$, this is true in~$K$. Hence, our definition 
of directed collapsibility preserves spaces of directed paths with an initial vertex 
of $\mathbf{0}$. This allows us to study more types of concurrent programs and 
preserve notions of partial executions.

There are many future avenues of research that we hope to pursue in the directed topological setting.  First, we hope to find necessary and sufficient conditions for a pair of cubical cells $(\tau,\sigma)$ to be a collapsing pair.   The key will be to have a better understanding of what removing a cubical cell does to the past link of a complex.  In addition to this, as there are many types of simplicial collapse, it would be interesting to see what the directed counterpart to each of these types of collapses might be.  For example, is there a notion of strong directed collapse?  As strong collapse also considers the link of a vertex, a consideration of how this extends to a directed setting seems natural.

 Next, there is more to learn about past link obstructions.  It is clear
 performing a directed collapse will not alter the space of directed paths of a
 Euclidean cubical complex; however, if we are unable to perform a directed
 collapse due to a past link obstruction, what does this say about the space of
 directed paths?  \thmref{obstructions} is a start in this direction for $0$ collapses.  Another question may be,  in what way are obstructions of type $\infty$ realized as
 non-contractible spaces of directed paths?

 Another direction of research we hope to pursue is defining a way to compute a directed homology that is collapsing invariant. Even in the two-dimensional setting (where the cubes are at most dimension two), this has proved to be difficult, as adding one two-cell can have various effects depending on the past links of the vertices involved. We would like to classify the spaces where such a dynamic programming approach would work.

Lastly, there are many computational questions on how to implement the collapse of a directed cubical complex. In \cite{lachaud}, an example of collapsing a three-dimensional cubical complex is implemented in C++. This could be used as a model when handling the directed complex.

Many interesting theoretical and computational questions continue to emerge in the field of directed topology. We hope that our research
excites others in studying cubical complexes in the directed setting.

\begin{acknowledgement}
This research is a product of one of the working groups at the
Women in Topology (WIT) workshop at MSRI in November 2017. This workshop was
organized in partnership with MSRI and the Clay Mathematics Institute, and was
partially supported by an AWM ADVANCE grant (NSF-HRD 1500481).  
In addition, LF and BTF further collaborated at the Hausdorff Research Institute
for Mathematics during the Special Hausdorff Program on Applied and
Computational Algebraic Topology (22017).

The authors also thank the generous support of NSF.
RB is partially supported by the NSF GRFP (grant no.\ DGE 1649608).
BTF is partially supported by NSF CCF 1618605.
CR is partially supported by the NSF GRFP (grant no.\ DGE 1842165).\end{acknowledgement}

\bibliographystyle{plain}
\bibliography{bibliography}

\begin{thebibliography}{10}

\bibitem{barmak-minian:2012}
Elias~Gabriel Barmak, Jonathan Ariel \&~Minian.
\newblock Strong homotopy types, nerves and collapses.
\newblock {\em E.G. Discrete Comput Geom}, 47(2):301--328, 2012.

\bibitem{Fajdianddi}
Lisbeth Fajstrup.
\newblock Dipaths and dihomotopies in a cubical complex.
\newblock {\em Advances in Applied Mathematics}, 35(2):188 -- 206, 2005.

\bibitem{lisbeth}
Lisbeth Fajstrup, Eric Goubault, Emmanuel Haucourt, Samuel Mimram, and Martin
  Raussen.
\newblock {\em Directed Algebraic Topology and Concurrency}.
\newblock Springer Publishing Company, Incorporated, 1st edition, 2016.

\bibitem{FGR}
Lisbeth Fajstrup, Eric Goubault, and Martin Raussen.
\newblock Algebraic topology and concurrency.
\newblock {\em Theoretical Comuter Science}, pages 241--271, 2006.

\bibitem{forman:2001}
Robin Forman.
\newblock A user's guide to discrete morse theory.
\newblock {\em S\' em. Lothar. Combin.}, 48, 12 2001.

\bibitem{hoare}
Charles~A.R. Hoare.
\newblock Communicating sequential processes.
\newblock {\em Commun. ACM}, 21(8):666--677, August 1978.

\bibitem{lachaud}
Jacques-Olivier Lachaud.
\newblock Cubical complex collapse, 2017.

\bibitem{raussen_2000}
Martin Raussen.
\newblock On the classification of dipaths in geometric models for concurrency.
\newblock {\em Mathematical Structures in Computer Science}, 10(4):427–457,
  2000.

\bibitem{RZ}
Martin Raussen and Krzysztof Ziemia{\'{n}}ski.
\newblock Homology of spaces of directed paths on euclidean cubical complexes.
\newblock {\em Journal of Homotopy and Related Structures}, 9(1):67--84, Apr
  2014.

\bibitem{whitehead}
John~H.C. Whitehead.
\newblock Simplicial spaces, nuclei and m-groups.
\newblock {\em Proceedings of the London Mathematical Society},
  s2-45(1):243--327, 1939.

\bibitem{ziemianski2016execution}
Krzysztof Ziemia{\'n}ski.
\newblock On execution spaces of {PV}-programs.
\newblock {\em Theoretical Computer Science}, 619:87--98, 2016.

\end{thebibliography}

\end{document}